\documentclass[a4paper,11pt]{amsart}
 \usepackage{amssymb,amsfonts,amsthm,amscd,stmaryrd}
\usepackage[utf8]{inputenc}
\usepackage[foot]{amsaddr}
\usepackage[dvipsnames]{xcolor} 
\usepackage{graphicx}
\usepackage{tikz}
\usetikzlibrary{decorations.pathreplacing}
\usepackage{subcaption}
\usepackage{color}

\usepackage{pgf,tikz}
 \usetikzlibrary{arrows}
\usepackage[colorlinks=true,urlcolor=blue,
citecolor=red,linkcolor=blue,linktocpage,pdfpagelabels,
bookmarksnumbered,bookmarksopen]{hyperref}
 
\usepackage{ae}
\usepackage[utf8]{inputenc}
\numberwithin{figure}{section}

\usepackage{mathrsfs,a4wide}
\usepackage{esint}
\usepackage{color}

\usepackage{import}

\usepackage[leqno]{amsmath}

\newtheorem{thm}{Theorem}[section]
\newtheorem{lem}[thm]{Lemma}
\newtheorem{prop}[thm]{Proposition}

\theoremstyle{definition}
\newtheorem{defn}[thm]{Definition}

\newtheorem{oss}[thm]{Remark}

\newtheorem{example}[thm]{Example}
\numberwithin{equation}{section}

\newcommand{\beq}{\begin{equation}}
\newcommand{\eeq}{\end{equation}}
\newcommand{\Div}{\operatorname{div}}

\def\R{\mathbb R}

\def\H{\mathcal H}
\def\S{\mathbb S}
\def\e{\varepsilon}
\def\pa{\partial}

\def\restrict#1{\raise-.5ex\hbox{\ensuremath|}_{#1}}
\def\ew{e^{w(|x|)}}

\title{Some weighted isoperimetric inequalities in quantitative form}

\author{Nicola Fusco \and Domenico Angelo La Manna}
\begin{document}
\begin{abstract}
In this paper we study two different weighted isoperimetric inequalities. First we prove a sharp stability result for the isoperimetric inequality with a log-convex weight. Then we analize the behavior of a negative power weight for the perimeter thus providing a complete picture of the isoperimetric problem in this context.  
\end{abstract}
\maketitle
\section{Introduction}
In recent years weighted isoperimetric inequalities have attracted the attention of many authors   (\cite{Mor},\cite{RCBM},\cite{CMV},\cite{mp},\cite{PrSa},\cite{PS}) also in view of their applications to different fields of Analysis. They play an important role in dealing with Gamow type energies, see for instance \cite{knmu13}, \cite{bocr}, \cite{knmu14}, \cite{ju14}, \cite{ffmmm15}, \cite{LAM1} and in shape optimization problems involving eigenvalues, see \cite{bdp},\cite{fnt},\cite{gavitone2020quantitative},\cite{cl} and the references therein.

Due to the relevance of the topic for applications, it would be important to understand stability properties of such inequalities.
Isoperimetric inequalities in quantitative form have a long history, see \cite{FMP}, \cite{FiMP}, \cite{CiLe}, and \cite{fusco} for a complete overview on the subject.

In this paper we focus our attention on two types of weighted isoperimetric problems: first we study the case of a log-convex  density $e^{w(|x|)}$ and then we consider a power density $|x|^p$ with $p$ negative.
While for the Gaussian isoperimetric inequality, where the density $e^{-|x|^2}$ is log-concave, it is well known that isoperimetric sets are half spaces (\cite{Bo}) and that they are stable (\cite{cfmp},\cite{bbj}, see also \cite{NPS,DeL,CCLP} for the non local extension),
the case of a log-convex weight has been only recently settled by G.R. Chambers in \cite{cha}.
In that paper he proved that if $w$ is a $C^3$ even convex function, then balls centered at the origin are the unique minimizers of the weighted perimeter under a weighted volume constraint. 
Note that in the Gaussian case it can be proved that balls with small weighted mass are local minimizers while this property fails if the mass is sufficiently large, see \cite{l}.

In this paper we study the stability of the isoperimetric problem with density $e^{w(|x|)}$.
For a set $E$ of locally finite perimeter we denote by $|E|_{w}$ and $P_w(E)$ respectively
$$
|E|_{w}= \int_{E} \ew \,dx\; \text{ and }\; P_{w}(E)= \int_{\pa^* E} \ew \,d\H^{n-1},
$$
where $\pa^* E$ is the so called reduced boundary of $E$, see \cite{maggi}.
We prove that the aforementioned isoperimetric inequality is actually stable. To be precise, our main theorem reads as follows.
Here and in the following we denote the ball of radius $r$ centered at $x$ by $B_r(x)$ or simply by $B_r$ when the center is the origin.  
\begin{thm}\label{thm:main}
Let $n\geq 2$. 
Given an even convex function $w\in C^3(\R)$ and $r>0$ such that $w''(r)>0$, there exists a constant $\kappa
=\kappa(n,r,w)$ such that, for any measurable set $E\subset\R^n$ with $|E|_w=|B_r|_w$, 
\beq\label{eq:mainth}
P_w(E)-P_w(B_r)\geq\kappa |E\triangle B_r|_w^2.
\eeq
\end{thm}

We stress that although the constant in \eqref{eq:mainth} might not be optimal, the exponent is sharp as can be seen by computing the value of the weighted perimeter on suitable ellipsoids, see for instance a similar example in Section 4 of \cite{fusco}.
Note also that differently from the quantitative isoperimetric inequality for the standard Euclidean perimeter, in our case we have no scaling properties and this explains why on the right hand side of \eqref{eq:mainth} the constant depends on $r$ and we have $|E\triangle B_r|_w$ instead of the usual Fraenkel asymmetry.

We now give a brief overview of the proof. Inspired by \cite{CiLe}, 
we use the so called selection principle. We first prove that \eqref{eq:mainth} holds true for a special class of sets, namely sets which are sufficiently close to the ball centered at the origin with the same weighted volume.
This first step of the proof is achieved by a Fuglede type argument.
Then we reduce by a compactness argument to the case where the set $E$ is close to such a ball in the $L^\infty$ sense. Precisely we argue by contradiction assuming that there exists a sequence $\{E_h\}_{h\in \mathbb N}$ such that $|E_h|_w=|B_r|_w$ for all $h$ and 
\eqref{eq:mainth}
fails for a suitably small constant.
Following an idea of \cite{AcFuMo} we
 construct a sequence of functionals $J_h$ whose minimizers $F_h$ also satisfy the opposite inequality in \eqref{eq:mainth} with a small constant and converge in $C^{1,\alpha}$ to $B_r$, thus getting a contradiction with the estimate proved in the first part of the proof. 
Although this type of argument has become more or less standard,
 one of the key difficulties here, due to the fact that the density diverges at infinity, is to show that the functionals $J_h$ do admit a minimizer and to get suitable a priori estimate ensuring the $C^{1,\alpha}$ convergence of $F_h$.
 
Another difficulty in this context, maybe the most challanging one, comes from the fact that neither the weighted volume nor the weighted perimeter are invariant under translations. This would make  a Fuglede type estimate for nearly spherical sets $E$ useless, since it usually requires that the barycenter of $E$ is at the origin. However, an interesting feature of our problem is that the assumption $w''(r)>0$ yields such a Fuglede type estimate without any further hypothesis on the barycenter,  see \eqref{eq:fuglede}.  Even more, this assumption turns out to be necessary for the validity of the quantitative inequality \eqref{eq:mainth}. More precisely,  the following result holds.
\begin{prop}\label{nece}
 Let $w\in C^2(\R)$ be a convex function such that  $w''(r)=0$ for some $r>0$. Then
$$
\lim_{\e\to0+}\frac{P_w(B_{\rho(\e)}(\e e_1))-P_w(B_r)}{|B_{\rho(\e)}(\e e_1)\triangle B_r|_w^2}=0\,,
$$
where $\rho(\e)>0$ is such that $|B_{\rho(\e)}(\e e_1)|_w=|B_r|$.
\end{prop}
The last part of the paper is devoted to another weighted isoperimetric inequality. This time we do not deal with a log-convex density. Instead, the weight is given by $|x|^p$, with $p$ negative.
While for $p>0$ the characterization of the balls centered at the origin as the unique isoperimetric sets  and their stability is well known, see  \cite{BeBrMePo}, \cite{Csa}, \cite{CGPRS},\cite{BrDeRu} and the references therein, the case 
$p<0$ is less understood.
First of all, notice that when $-n+1<p<0$ then the problem becomes trivial since for any fixed mass the infimum of the weighted perimeter under the mass constraint is 0. On the other hand, it is known (see \cite{Csa}) that if $p\leq 1-n$ and $E$ is a bounded open set with Lipschitz boundary containing the origin, then 
\beq \label{eq:isop}
\int_{\pa^* E}|x|^p \, d\H^{n-1}\geq \int_{\pa B_r}|x|^p \, d\H^{n-1},
\eeq
where $B_r$ has the same volumeof $E$. Moreover, if equality holds in \eqref{eq:isop}, $E$ coincides with $B_r$. Note that the assumption that the origin belongs to the interior of $E$ is crucial, since it can be easily checked that \eqref{eq:isop} may fail when the origin belongs to the interior of the complement of $E$.
In the last section of the paper we extend \eqref{eq:isop} to the case of a set  $E$ of locally finite perimeter such that $0 \in \operatorname{int}(E^{(1)})$, where
$E^{(1)}$ is the set of points where $E$ has density 1. 
Note that if $0 \in \pa^* E$ \eqref{eq:isop} becomes
trivial because in this case the left hand side is infinite, see Remark \ref{remrem}. 
Note also that the assumption that $0 \in \operatorname{int}(E^{(1)})$ is sharp in the sense that one may construct a set of finite perimeter $E$ such that $0\in \pa E^{(1)}\setminus \pa^* E $ for which the inequality fails.
Finally we prove that \eqref{eq:isop} also holds in a quantitative form. 
\section{Notation and Preliminary Results}
Throughout all the paper we will assume that $w: \R\to \R$ is a $C^3$ even convex function. 
In the sequel by $C,c$ we denote positive constants whose value may change from line to line and occasionally we highlight the dependence of these constants by other parameters. The dependence of the constants on $w$ will be always tacitly understood.   
For $n\geq 2$ let $E \subset \R^n$ be a measurable set and $\Omega\subset \R^n$ an open set. We say that 
$E$ has finite weighted perimeter with respect to $e^{w(x)}$ in $\Omega$ if 
\[
P_w(E;\Omega)=\sup_{\|X\|_{L^\infty(\R^n)}\leq 1}\left\{\int_{\Omega}\Div(e^{w(x)} X(x))\, dx, \,\, X\in C^1_c(\R^n;\R^n)\right\}.
\]
From this definition it follows that if a set has finite weighted perimeter in $\Omega$ then
$P(E;\Omega)<\infty$, where $P(E;\Omega)$ denotes the standard Euclidean perimeter in $\Omega$.
If $\Omega=\R^n$ we simply write $P_w(E)$ or $P(E)$ in place of $P_w(E;\R^n)$ and $P(E;\R^n)$. For the definitions and properties of sets of finite perimeter we refer to \cite{maggi}.
Note that if $\pa^* E$ is the reduced boundary of $E$ from the definition above we have\[
P_w(E;\Omega)=\int_{\pa^* E\cap \Omega} e^{w(x)}\, d\H^{n-1},
\]
where $\H^{n-1}$ denotes the $(n-1)$-dimensional Hausdorff measure.
We recall that at every point $x\in \pa^* E$ the exterior generalized normal $\nu_E(x)$ is defined and the following generalized Gauss-Green formula holds
\[
\int_{E} \Div X \, dx= \int_{\pa^* E} \langle X,\nu_E \rangle\, d\H^{n-1}
\]
for every vector field $X\in C_c^1(\Omega,\R^n)$. 
We now introduce the notion of quasiminimizer of the perimeter.
\begin{defn} \label{defn:quasimin}
Let $E\subset \R^n$ be a set of locally finite perimeter, $\omega\geq0$ and let $\Omega$ be an open subset of $\R^n$. We say that $E$ is an $\omega $-minimizer of the perimeter in $\Omega$ if for every ball $B_\rho(x)\subset \subset\Omega $ with $\rho<1$ and for any set $F$ of locally finite perimeter such that $E\triangle F\subset\subset B_\rho (x)$  it holds
\beq \label{eq:disquasi}
P(E;B_\rho(x))\leq P(F;B_\rho(x))+\omega \rho^n
\eeq
\end{defn}
The following theorem is a consequence of the classical De Giorgi's $\e$-regularity theorem, see for instance \cite[Theorem~1.9]{tamanini}
and also the argument of the proof of \cite[Lemma~3.6]{cl}.
Before stating it we recall that if $E_h$ and $E$ are measurable sets of $\R^n$ and $\Omega$ is an open set, one says that $E_h\to E$ in measure in $\Omega$ if   $|E_h \triangle E \cap \Omega|\to 0 $. The local convergence in measure is defined in the obvious way. 

\begin{thm}\label{teo:quasimin2}
Assume that $E_h,E$ are equibounded $\omega$-minimizers of the perimeter in $\R^n$ such that $E_h\to E$ in measure. If $E$ is of class $C^2$ then for $h$ large $E_h$ is of class $C^{1,\frac12}$ and there exists a function $v_h:\partial E\to\R$ such that
\[
\pa E_h= \{x+ v_h(x)\nu_E(x), \,\, x \in \pa E \}.
\]
Moreover, $\|v_h\|_{C^{1,\alpha}(\pa E)}\to 0$ for $0<\alpha <\frac12$.
\end{thm}
\section{A first stability estimate}

In this section we give a Fuglede type result for nearly spherical sets under the assumption that $w''(r)>0$. The interesting feature of this result is that this assumption allows to prove the quantitative estimate \eqref{eq:fuglede} without any assumption on the barycenter of $E$.
We start with the definition of nearly spherical sets.
\begin{defn}\label{defn:nearlysph} 
Let $n\geq 2$. We say that a set $E$ is nearly spherical if there exist $r>0$ and a Lipschitz function $u:\S^{n-1}\to (-1,1)$ such that 
$$
E=\{y:y=r x(1+u(x)), \, x\in \mathbb S^{n-1}\}.
$$
\end{defn}
\begin{prop} \label{lem:nearlysph}
 Given $0<r_1<r_2$ such that $w''(r)>0$ for all $r\in [r_1,r_2]$, there exist 
 $\e, c>0$ such that if
 $E$ is a nearly spherical set as in Definition \ref{defn:nearlysph} with $\|v\|_{W^{1,\infty}(\S^{n-1})}<\e$, 
$|B_r|_w=|E|_w$, $r\in [r_1,r_2]$, 
 then
 \beq \label{eq:fuglede}
 P_w(E)-P_w(B_r) \geq cr^{n-1}e^{w(r)} \|\nabla _\tau u\|_{L^2(\S^{n-1})}^2,
 \eeq
 where $\nabla_\tau u$ stands for the tangential gradient of $u$. 
 \end{prop}
\begin{oss}
Note that by the Poincar\'e inequality on the sphere, \eqref{eq:fuglede} implies
\beq \label{eq:fuglede1}
P_w(E)-P_w(B_r) \geq cr^{n-1}e^{w(r)} \| u\|_{L^2(\S^{n-1})}^2\geq c'r^{n-1}e^{w(r)}|E\triangle B_r|^2_w.
\eeq
However \eqref{eq:fuglede} is clearly stronger than the quantitative inequality \eqref{eq:mainth}.
\end{oss}
\begin{proof}[Proof of Lemma \ref{lem:nearlysph}]
Given $r\in [r_1,r_2]$ we consider the Lipschitz map $\Psi: \overline{B}\to  E$ defined by $\Psi(x)= r x\left(1+u\left(\frac{x}{|x|}\right)\right)$. A straightforward computation shows that for every $x \in B$ the Jacobian $J\Psi$ at $x$ is given by 
 \beq \label{eq:jacobian}
 J\Psi (x) =r^n \left(1+u\left(\frac{x}{|x|}\right)\right)^n
 \eeq
 while for $x \in\S^{ n-1}$ the tangential Jacobian is 
$$
 J_\tau \Psi  (x) =r^{n-1} (1+u(x))^{n-2} \sqrt{(1+u)^2+|\nabla_\tau u|^2}.
$$
From this last equality, using the area formula, we get
$$
P_w(E)=r^{n-1}\int_{\S^{n-1}} (1+u(x))^{n-1}\sqrt{1+\frac{|\nabla_\tau u|^2}{(1+u)^2}}e^{w(r(1+u(x)))}\, d\H^{n-1}
$$
and recalling \eqref{eq:jacobian} 
$$
|E|_w= r^n\int_{\S^{n-1}}(1+u(x))^n\int_0^1 t^{n-1}e^{w(rt(1+u(x))}\, dt\,dx.
$$
Hence
\[
\begin{split}
 P_w(E)-P_w(B_r)
 =& r^{n-1}\int_{\S^{n-1}}\left(( 1+u(x))^{n-1}\sqrt{1+\frac{|\nabla_\tau u|^2}{(1+u)^2}}e^{w(r(1+u(x)))}- e^{w(r)})\right)\, d\H^{n-1}    
\\
=& r^{n-1}\int_{\S^{n-1}}( 1+u(x))^{n-1}\left(\sqrt{1+\frac{|\nabla_\tau u|^2}{(1+u)^2}}-1\right)e^{w(r(1+u(x)))}\, d\H^{n-1}
\\
&+r^{n-1}\int_{\S^{n-1}}(1+u(x))^{n-1}e^{w(r(1+u(x)))} -  e^{w(r)}\,d\H^{n-1}:= r^{n-1}(I_1+I_2)
\end{split}
\]
Using the hypothesis $\|u\|_{W^{1,\infty}}<\e$ we can control the term $I_1$ by 
\[
I_1 \geq \left(\frac12-C\e\right) e^{w(r)} \int_{\S^{n-1}}|\nabla_\tau u|^2 d\H^{n-1},
\]
for some constant $C>0$ uniformly bounded for $r\in[r_1,r_2]$.
To estimate $I_2$, by a second order Taylor expansion we have
\[\begin{split}
    \frac{(1+u(x))^{n-1}e^{w(r(1+u(x)))} -  e^{w(r)}}{e^{w(r)}}= (n-1 + r w'(r))u \\+ \frac12((n-1)(n-2) +2(n-1)rw'(r)+r^2 w''(r) +r^2 w'(r)^2        )u^2 +o(u^2).
\end{split}
\] 
From this equality and the estimate on $I_1$ we get
\beq\label{eq:area0}\begin{split}
   &\frac{ P_w(E)-P_w(B_r)}{r^{n-1} e^{w(r)}}\geq 
\left(\frac12-C\e\right)  \int_{\S^{n-1}}|\nabla_\tau u|^2 d\H^{n-1}+ (n-1+rw'(r))\int_{\S^{n-1}}u d\H^{n-1} \\
&\;\;\;\;\;\;+\frac12((n-1)(n-2) +2(n-1)rw'(r)+r^2 w''(r) +r^2 w'(r)^2-C\e)\int_{\S^{n-1}} u^2 d\H^{n-1},
\end{split}
\eeq
for a constant $C$ uniformly bounded for $r\in[r_1,r_2]$. Since $|E|_w=|B_r|_w$
we have
\[
\int_{\S^{n-1}}(1+u(x))^n\int_0^1 t^{n-1}e^{w(rt(1+u(x))}\, dt\,d\H^{n-1}= \int_{\S^{n-1}}\int_0^1 t^{n-1}e^{w(rt)}\, dt\,d\H^{n-1}
\]
which in turn gives
\[
\int_0^1t^{n-1}\int_{\S^{n-1}}\left((1+u(x))^n e^{w(rt(1+u(x)))}- e^{w(rt)}\right)\,d\H^{n-1}\, dt=0.
\]
By a second order Taylor expansion of the functions $(1+\cdot)^n$ and $e^{w(rt(1+ \cdot))}$, using again the smallness assumption on $u$, we get
\beq\label{eq:area}\begin{split} 
(na_n+rb_n)\int_{\S^{n-1}}
u \,d\H^{n-1}\,dt \geq &-\frac12 (n(n-1)a_n +2nrb_n+r^2 c_n +r^2d_n)\int_{\S^{n-1}} u^2 \,d\H^{n-1}\,dt\\
&-C\e   \int_{\S^{n-1}}u^2\, d\H^{n-1} ,
\end{split}
\eeq
where
\[
a_n=\int_0^1 t^{n-1}e^{w(rt)}\,dt,\;\;\;b_n= \int_0^1 t^{n}w'(rt)e^{w(rt)}\,dt,\;\; c_n= \int_0^1 t^{n+1}w''(rt)e^{w(rt)}\,dt
\]
and \[d_n=\int_0^1t^{n+1}(w'(rt))^2e^{w(rt)}\, dt.
\]
Integrating by parts we have the following identities
\[
rb_n= e^{w(r)}- na_n,\;\;\;r^2(c_n+d_n)= rw'(r)e^{w(r)}-(n+1)(e^{w(r)}-na_n),
\]
which in turn imply that \eqref{eq:area} can be rewritten as
\beq \label{eq:area1}
(n-1+rw'(r))\int_{\S^{n-1}}u \, d\H^{n-1}\geq-\left(\frac{(n-1+ rw'(r))^2}{2}-C\e\right)
\int_{\S^{n-1}}u^2 \, d\H^{n-1},
\eeq
where again $C$  depends only on $r_1,r_2$.

Note that if $\int_{\S^{n-1}}u\, d\H^{n-1}\geq0$ then \eqref{eq:fuglede} follows at once from \eqref{eq:area0} with $c=\frac14$, provided $\e$ is small enough.
Assume instead that
\beq\label{nega}
\int_{\S^{n-1}}u\, d\H^{n-1}<0.
\eeq
Collecting all the previous inequalities we have
\beq \label{eq:nearlysph}
\begin{split}
    \frac{P_w(E)-P_w(B)}{r^{n-1}e^{w(r)}}\geq& \left(\frac12-C\e\right) \|\nabla u\|_{L^2(\S^{n-1})}^2\\
    &+\frac12\left(1-n+r^2w''(r)-C\e \right)\|u\|_{L^2(\S^{n-1})}^2.
\end{split}
\eeq
For $k\in \mathbb N$ and $i\in \{1,\dots, G(n,k)\}$ let $Y_{k,i}$ be the spherical harmonics of order $k$, i.e., the restrictions to $\S^{n-1}$ of the homogeneous harmonic polynomials of degree $k$, normalized so that $\|Y_{k,i}\|_{L^2(\S^{n-1})}=1$. Note that $\{Y_{k,i}\}_{k\in \mathbb N, i\leq G(n,k)}$ forms an orthonormal system for $L^2(\S^{n-1})$ and that for every $k,i$  
\[
-\Delta_{\S^{n-1}} Y_{k,i}=k(k+n-2) Y_{k,i},
\]
where $\Delta_{\S^{n-1}}$ is the Laplace-Beltrami operator on $\S^{n-1}$. 
Hence, we may represent $u$ with respect to this orthonormal system as
\[
u=\sum_{k=0}^\infty\sum_{i=1}^{G(n,k)} a_{k,i}Y_{k,i},\quad
\text{where }\quad a_{k,i}=\int_{\S^{n-1}}u Y_{k,i} \,d\H^{n-1},
\]
thus getting
\[
\|u\|_{L^2(\S^{n-1})}^2=\sum_{k\geq 0}\sum_{i=1}^{G(n,k)} a_{k,i}^2
\quad \text{and} \quad \|\nabla_\tau u\|_{L^2(\S^{n-1})} =\sum_{k\geq 1}\sum_{i=1}^{G(n,k)} k(k+n-2)a^2_{k,i}.
\]
Note that  condition \eqref{eq:area1}, together with \eqref{nega},   implies that
\[
a_{0,1}^2\leq C\e \|u\|_{L^2(\S^{n-1})}^2.
\]
Observe that if $1-n+r^2w''(r)> 0$
the conclusion holds with $c=1/4$ and $\e$ sufficiently small. Otherwise,  
\[
\|\nabla_\tau u\|^2_{L^2(\S^{n-1})}
=\sum_{k\geq 1}\sum_{i=1}^{G(n,k)} k(k+n-2)a^2_{k,i}
\geq (n-1)\|u\|^2_{L^2(\S^{n-1})} -C\e\|u\|_{L^2(\S^{n-1})}^2.
\]
Employing the latter in \eqref{eq:nearlysph} we infer
\[
 P_w(E)-P_w(B)\geq  \frac12 r^{n-1}e^{w(r)}\left(r^2w''(r) -C\e\right) \|\nabla_\tau u\|_{L^2(\S^{n-1})}
\]
which concludes the proof of \eqref{eq:fuglede} taking $\e$ small enough.
\end{proof}

Note that a suitable modification  of the above proof, see for instance the proof of \cite[Th. 3.1]{fusco} immediately yields that if $w''(r)=0$ then  \eqref{eq:fuglede} still holds under the assumption that the barycenter of $E$ is at the origin. However, if this condition is not satisfied and $w''(r)=0$, then not only \eqref{eq:fuglede}, but even \eqref{eq:mainth} is not longer true, see Proposition~\ref{nece}. 

In order to prove this, given $r>0$, for any $\e>0$ we denote by $\rho(\e)$ the unique positive number such that
$$
|B_r|_w=|B_{\varrho(\e)}(\e e_1)|_w.
$$
Since $w$ is of class $C^2$, it is easily checked that  $\rho\in C^2(0,\infty)$.
\begin{proof}[Proof of Proposition~\ref{nece}]
To simplify the notation we assume, without loss of generality, that $r=1$ and set $B=B_1$.
 For  $\e>0$ let $\rho =\rho (\e)$ such that  
\beq \label{eq:opt}
|B|_w=|B_{\rho(\e)}(\e e_1)|_w.
\eeq 
 Clearly, $\rho (0)=1$. Differentiating \eqref{eq:opt} with respect to $\e$ we get
\beq \label{eq:opt1}
\begin{split}
0&=\frac{d}{d\e}|B_{\rho(\e)}(\e e_1)|_w= \frac{d}{d\e}\left(\rho(\e)^{n}\int_{B}e^{w(|\rho(\e)x +\e e_1|)} \, dx\right)
\\
=& n \rho(\e)^{n-1}\rho'(\e)\int_{B}e^{w(|\rho(\e)x +\e x_1|)} \, dx
\\
&+ \rho(\e)^n\int_{B}e^{w(|\rho(\e)x +\e e_1|)} w'(|\rho(\e)x +\e e_1|)\Big\langle \frac{\rho(\e)x+\e e_1}{|\rho(\e)x+\e e_1|}, \rho'(\e)x+e_1 \Big\rangle\, dx.
\end{split}
\eeq
By symmetry
\[
\int_{B}e^{w(|x|)}  w'(|x|)\frac{x_1}{|x|} \, dx=0.
\]
Therefore, evaluating \eqref{eq:opt1} at $\e=0$, we get
\[
\rho'(0)\left(n|B|_w +\int_{B}e^{w(|x|)}|x|w'(|x|)\, dx\right) =0
\]
which implies $\rho'(0)=0$.
Differentiating \eqref{eq:opt1} again with respect to $\e$ and evaluating the second derivative at $\e=0$ we also get
\[
\begin{split}
&\rho''(0)\bigg(n|B|_w+\int_{B} e^{w(|x|)} |x|w'(|x|)
 \, dx\bigg) + \int_B e^{w(|x|)} \frac{x_1^2}{|x|^2}w'^2(|x|)\, dx\\
& \qquad + \int_B e^{w(|x|)} \frac{x_1^2}{|x|^2}w''(|x|)\, dx+\int_{B}e^{w(|x|)}\frac{w'(|x|)}{|x|} \, dx - \int_{B}e^{w(|x|)}\frac{x_1^2}{|x|^3}w'(|x|)\, dx=0.
\end{split}
\]
Thus after some simplifications
\beq \label{eq:opt2}
\rho''(0)= -\frac{1}{n}\frac{\displaystyle\int_Be^{w(|x|)}\Big(w'^2(|x|)+w''(|x|)+(n-1)\frac{w'(|x|)}{|x|}\Big)\,dx }{n|B|_w+\displaystyle\int_{B} e^{w(|x|)} |x|w'(|x|)
 \, dx}.
\eeq
A similar calculation shows that $\frac{d}{d\e}\Big(P_w(B_{\rho(\e)}(\e e_1))\Big)_{\big|_{\e =0}}=0$ and 
\[
\begin{split}
&\frac{d^2}{d\e^2}\left(P_w(B_{\rho(\e)}(\e e_1))\right)_{\big|_{\e =0}} \\
&\qquad\qquad = P_w(B)\Big[\rho''(0)(n-1+w'(1))+ \frac{1}{n}\Big(w'^2(1)+w''(1)+(n-1)w'(1)\Big)\Big].
\end{split}
\]
From this equation, using \eqref{eq:opt2}, we get, recalling that $w''(1)=0$,
\beq \label{eq:opt3}
\begin{split}
&\frac{1}{P_w(B)}\frac{d^2}{d\e^2}\left(P_w(B_{\rho(\e)}(\e e_1))\right)_{\big|_{\e =0}}=\frac{1}{n}\Big(w'^2(1)+(n-1)w'(1)\Big) \\
& -\frac{(n-1+w'(1))}{n}\frac{\displaystyle\int_Be^{w(|x|)}\Big(w'^2(|x|)+w''(|x|)+(n-1)\frac{w'(|x|)}{|x|}\Big)\,dx }{n|B|_w+\displaystyle\int_{B} e^{w(|x|)} |x|w'(|x|)
 \, dx}. \end{split}
\eeq
Observe that by divergence theorem
\[
\begin{split}
(n-1)\int_{B}e^{w(|x|)}\frac{w'(|x|)}{|x|} \, dx &= \int_B e^{w(|x|)}w'(|x|)\Div \Big(\frac{x}{|x|}\Big) \, dx\\
&= w'(1)P_w(B)-\int_{B}e^{w(|x|)}\big(w'^2(|x|)+w''(|x|)\big)\,dx
\end{split}
\]
and similarly
\[
\int_{B} e^{w(|x|)} |x|w'(|x|)
 \, dx=  \int_{B}\Div (xe^{w(|x|)})\, dx - n |B|_w=
P_w(B_1)- n|B|_w.
\]
Plugging these identities in \eqref{eq:opt3}  we have that $\frac{d^2}{d\e^2}\left(P_w(B_{\rho(\e)}(\e e_1))\right)_{\big|_{\e =0}}=0$.
Thus $P_w(B_{\rho (\e)}(\e e_1))=P_w(B)+o(\e^2)$ and the conclusion follows.
\end{proof}

\section{Preliminary Lemmas} 
We start this section by recalling  the result of \cite{cha} on the uniqueness of balls as isoperimetric sets for the log-convex isoperimetric inequality. We recall this result for the reader's convenience.
\begin{thm}\label{teo:chamb}
If $w$ is a convex even function of class $C^3$ with $w(r)>w(0)$ for $r>0$ the only isoperimetric regions are balls centered at the origin.
\end{thm}
Next lemma shows the continuity of  $P_w(\cdot)$ at $B_r$ with respect to the convergence in measure. 
\begin{lem}\label{lem:qualitative}
Let $r>0$ such that $w(r)>w(0)$. Given $\e>0$ there exists $\delta>0$ such that for every set of finite perimeter $E$ with $|E|_w=|B_r|_w$, if $P_w(E)-P_w(B_r)<\delta$ then $|E\triangle B_r|_w < \e$.
\end{lem}
\begin{proof}
Assume by contradiction that there exists $\e_0>0$ such that for every $k\in \mathbb N$ there exists a set $E_k$ with $|E_k|_w=|B_r|_w=m$ and such that $P_w(E_k)-P_w(B_r)\leq \frac1k$, but $|E_k\triangle B_r|_w\geq \e_0$.
Since $w$ is increasing on $\R^+$, for $k$ sufficienlty large we have 
$$
e^{w(0)}P(E_k)\leq P_w(E_k) \leq 2 P_w(B_r).
$$
Hence $\{E_k\}_{k\in \mathbb N}$ is a sequence of sets with equibounded perimeters and thus, up to a not relabeled subsequence, we have that there exists a set $E$ such that
$\chi_{E_k}\to\chi_E$ in $L^1_{\text{loc}}(\R^n)$ and $$P_w(E) \leq \liminf_k P_w(E_k)=P_w(B_r).$$
We claim that $|E|_w=m$.

To this aim it is enough to show that given  $\sigma>0$ there exists $R>0$ such that $|E_k \setminus B_R|_w<\sigma$
for all $k$.
Indeed, if there exists $k_0$ such that $|E_{k_0}\setminus B_R|>\sigma$ then 
\beq\label{eq:coarea}
|E_{k_0}\setminus B_R|_w=\int_{R}^\infty \H^{n-1}(E_{k_0}\cap \pa B_t) e^{w(t)}\, dt>\sigma.
\eeq
Recall that, for a.e. $t>0$, $E_{k_0}\cap \pa B_t$  is a set of finite perimeter on the sphere such that
$\pa^* E_k \cap \pa B_t$ coincides up to a set of zero $\H^{n-2}$-measure with the reduced boundary of $ E_{k_0}\cap \pa B_t$ relative to $\pa B_t$, see for instance \cite[Theorem~3.7]{CaPeSt}. 
If $G\subset \S^{n-1}$ is a set of finite perimeter denote by $\pa_{\S^{n-1}} G$ the boundary of $G$ relative to $\S^{n-1}$ and by $\pa^*_{\S^{n-1}}G$ the corresponding reduced boundary relative to $\S^{n-1}$. Then, the isoperimetric inequality on the sphere (see \cite{BoDuFu}) states that
\begin{equation}\label{eq:isopspher}
\H^{n-2}(\pa^{*}_{\S^{n-1}} G) \geq \H^{n-2}(\pa_{\S^{n-1}} S_\theta)
\end{equation}
where $S_\theta$ is the spherical cap 
with geodesic radius $\theta$ such that
$\H^{n-1}(S_\theta)=\H^{n-1}(G)$.
Since
\[
\H^{n-1}(S_\theta)=(n-1)\omega_{n-1}\int_{0}^\theta \sin^{n-2} \varphi\, d\varphi,\qquad \H^{n-2}(\pa_{\S^{n-1}} S_\theta)= (n-1)\omega_{n-1}\sin^{n-2} \theta,
\] 
a straightforward computation shows that  \eqref{eq:isopspher} implies that there exists $c_n>0$ such that
\beq \label{eq:isopspher1}
\H^{n-2}(\pa^{*}_{\S^{n-1}} G)\geq c_n  \left(\H^{n-1}(G)\right)^\frac{n-2}{n-1} \quad\text{whenever} \,\,\H^{n-1}(G)\leq\frac12\H^{n-1}(\S^{n-1}).
\eeq
Since for $R>0$ sufficiently large and for a.e. $t>R$ we have
\beq \label{eq:isopspher2}
\H^{n-1}(E_{k_0}\cap \pa B_t)\leq P(E_{k_0}; \R^{n}\setminus B_t)\leq \frac{1}{e^{w(R)}}P_w(E_{k_0})\leq\frac12\H^{n-1}(\S^{n-1}),
\eeq
from \eqref{eq:isopspher1}, using the coarea formula, we get 
\[\begin{split}
P_w(E_{k_0})\geq& 
\int_{R}^\infty \H^{n-2} ( \pa^* E_{k_0}\cap \pa B_t) \, dt
\geq
c_n\int_R^\infty \left(\H^{n-1}(E_{k_0}\cap \pa B_t)\right)^\frac{n-2}{n-1} e^{w(t)}\, dt 
\\
=& c_n \int_R^\infty\frac{\H^{n-1}(E_{k_0}\cap \pa B_t)}{\H^{n-1}(E_{k_0} \cap\pa B_t)^{\frac{1}{n-1}}}e^{w(t)}\, dt.
\end{split}
\]
From this inequality, using \eqref{eq:isopspher2} again and  recalling \eqref{eq:coarea} we conclude that
\[
P_w(E_{k_0})\geq c_n \sigma \left(\frac{e^{w(R)}}{P_w(E_{k_0})}\right)^{\frac{1}{n-1}}
\]
that is $P_w(E_{k_0})\geq c \sigma^{\frac{n-1}{n}}e^{\frac{w(R)}{n}}$ which is impossible if $R$ is sufficiently large.
This proves the claim, hence by Theorem \ref{teo:chamb} $E$ must coincide with a ball $B_r$, which is a contradiction since $|E_k\triangle B_r|_w \to 0$.
\end{proof}
Next simple lemma is a useful tool in the proof of the main theorem.
\begin{lem} \label{lem:uniqueness}
Let $r>0$ such that $w(r)>w(0)$, $\Lambda_1\geq 0$ and  $\Lambda_2\geq 2(4\frac{n+1}{r}+w'(2r))$. Then $B_r$ is the only minimizer of the functional defined for a measurable set $E\subset \R^n$ as
\[
J_{\Lambda_1,\Lambda_2}(E)=P_w(E) +\Lambda_1\left||E|_w-|B_r|_w\right|
+\Lambda_2|E\triangle B_r|_w.
\]
The same conclusion holds if $\Lambda_2=0$ and $\Lambda_1\geq n-1+rw'(r)$.
\end{lem}
\begin{proof}
Let $\eta: \R\to [0,1] $ be a smooth cut-off function such that
 $\eta(t)=1$ for $t\in \left[\frac r2, \frac{3r}{2}\right]$, $\eta(t)=0$ outside of the interval $[\frac r4,\frac{7r}{4}]$
and $\|\eta'(t)\|_{L^\infty}\leq 8/r$. 
Consider the smooth vector field $X(x)=\eta (|x|) \frac{x}{|x|}$.
It is readily checked that $\|X\|_{L^\infty}
=1$ and $\|\Div X\|_{L^\infty}\leq (4n+4)/r$

By definition of reduced boundary we get
\[\begin{split}
\int_{\pa^* E} e^{w(|x|)}\, d\H^{n-1}
\geq &\int_{\pa^*E} e^{w(|x|)}\langle X,\nu_E\rangle\, d\H^{n-1}=\int_{E}\Div(e^{w(|x|)}X)\, dx
\\= &
\int_{E} \left(\Div X + w'(|x|)\frac{\langle X,x \rangle }{|x|}   \right)    e^{w(|x|)}dx
\end{split}
\]
while for the ball it holds
\[
\int_{\pa B_r}e^{w(|x|)}\,d\H^{n-1}=\int_{B_r} \left(\Div X + w'(|x|)\frac{\langle X,x \rangle }{|x|}   \right)    e^{w(|x|)} \, dx.
\]
Hence we find
\[\begin{split}
  J_{\Lambda_1,\Lambda_2} (E)-
J_{\Lambda_1,\Lambda_2} (B_r ) \geq& \left(\Lambda_2 - \|\Div X\|_{L^\infty( \R^n)}- \|w'(|x|)X\|_{L^\infty(\R^n)}      \right)|E\triangle B_r |_w\\
\geq & \left( \Lambda_2 -\frac{4n+4}{r}- w'(2r) \right)|E\triangle B_r |_w
.
\end{split}
\]
Taking in mind the definition of $\Lambda_2$ we immediately get the desired result.

If $\Lambda_2=0$ by the uniqueness result stated in Theorem \ref{teo:chamb}
we immediately get that the minimizers of $J_{\Lambda_1,\Lambda_2}$ are given by balls centered at the origin. 
On such balls the value of the functional is given by
\[
J_{\Lambda_1,\Lambda_2}(B_\varrho)= n\omega_n \varrho^{n-1}e^{w(\varrho)}+ n\omega_n\Lambda_1\left| 
\int_\varrho^r e^{w(t)} t^{n-1}\,dt\right|=f(\varrho).
\]
By an elementary computation we get that under the assumption on $\Lambda_1$ the function $f(\varrho)$ attains its unique minimum when $\varrho=r$. 
\end{proof}
\section{Proof of theorem \ref{thm:main}} \label{sec:proof}
This section will be devoted to the proof of Theorem \ref{thm:main} which is achieved by a contradiction argument which makes use of suitable energy functionals. One problem here is to show the existence of minimizers for such functionals. This fact is achieved by showing that there exists a minimizing sequence made up by equibouded sets. 

%
%
%
To this aim we introduce the functions  $\Phi,\Psi: \R_+\to \R_+$ defined for $s,t\geq0$ as   
\beq \label{eq:psi}
\Phi(s)= n\omega_n\int_0^s t^{n-1}e^{w(t)}\, dt,\qquad \Psi(t)= \Phi^{-1}(t).
\eeq
Note that $\Psi$ is well defined since $\Phi$ is a strictly increasing function. Note also that $\Psi(t)$ is equal to the radius $r$ of the ball centered at the origin such that $|B_{r}|_w=t$.
The following lemma contains a few useful properties of $\Psi$ whose elementary verification is left to the reader.

\begin{lem}\label{lem:psi}
Let $\Psi$ be the function defined in \eqref{eq:psi}. 
Then $\Psi \in C^\infty (0,\infty)$ and for $t>0$ we have
 \beq
 \begin{split} \label{eq:tecnico}
  \Psi'(t)&= \frac{1}{n\omega_n \Psi^{n-1}(t)e^{w(\Psi(t))}},
\\
 t&\leq n\omega_n \Psi(t)^n e^{w(\Psi(t))}.  
 \end{split}
   \eeq 

\end{lem}
\begin{lem}\label{lem:bound1}
Let $E\subset \R^n$ 
be a set of finite perimeter such that $|E\setminus B_r|_w\leq \eta<1$.
There exists $R_E \in [r, r +4\Psi(\eta)]$ such that 
\[
P_w(E)\leq P_w(E\cap B_{R_E}  )- \frac{|E\setminus B_{R_E} |_w}{2\Psi(\eta)}.
\]
\end{lem}
\begin{proof}
We argue by contradiction assumig that for any  $r\leq t\leq r +4\Psi(\eta)$ it holds
\beq \label{eq:bound}
P_w(E\cap B_{t})> P_w(E)-\frac{|E\setminus B_t|_w}{2\Psi(\eta)}.
\eeq
Set $v(t)= |E\setminus B_t |_w$. Then for a.e. $t>0$
$$
v'(t)=- e^{w(t)}\H^{n-1}(E \cap \pa B_t).
$$
Since $P_w(E)\geq P_w(E\cap B_t )+ P_w(E\setminus B_t )+2 v'(t) $,  inequality \eqref{eq:bound} implies that
\[
2v'(t)+  P_w(E\setminus B_t) < \frac{v(t)}{2\Psi(\eta)}.
\]
The weighted isoperimetric inequality hence gives
\[
2v'(t)+ n\omega_n\Psi(v(t))^{n-1}e^{w(\Psi(v(t)))} < \frac{v(t)}{2\Psi(\eta)}.
\]
We now use the second inequality in \eqref{eq:tecnico} to infer 
$$\frac{v(t)}{\Psi(\eta)}\leq   n\omega_n\Psi(v(t))^{n-1}e^{w(\Psi(v(t)))},
$$ 
which gives
\[
\Psi(v(t))^{n-1}e^{w\Psi(v(t))} <- \frac{4}{n\omega_n}v'(t) \qquad \text{for all $\,\,t\in[r,r+4\Psi(\eta)]$.}
\]
Integrating the latter inequality we get by a change of variable and using the first equality in \eqref{eq:tecnico} 
\[\begin{split}
4\Psi(\eta)&< -\frac{4}{n\omega_n}\int_{r}^{r+4\Psi(\eta)} \frac{v'(t)}{\Psi(v(t))^{n-1}e^{\Psi(v(t))} }dt
\\
&=
\frac{4}{n\omega_n} \int_{v(r+4\Psi(\eta))}^{v(r)}\frac{1}{\Psi(s)^{n-1}e^{\Psi(s)}}\, ds
\\
 &=4(\Psi (v(r))-\Psi(v(r+4\Psi(\eta))),
\end{split} 
\]
 which is impossible.
\end{proof}
We are now ready to state the following existence result. 
\begin{lem}     \label{lem:existence}
Let $r>0$ such that $w(r)>w(0)$, $\Lambda_1 \geq n-1+rw'(r)$ and $\Lambda_2> 0$. There exist $0<\alpha_1< \frac{\Lambda_2}{2\Lambda_2+1}$ such that for any $\alpha\in [0,\alpha_1]$ the functional
\[
J_{\Lambda_1,\Lambda_2,\alpha}(F)=
P_w(F)+\Lambda_1\left||F|_w-|B_r|_w\right|+\Lambda_2||F\triangle B_r|_w-\alpha|, \qquad F\subset \R^n,
\]
has always a minimizer $E\subset B_{R_0}$
where 
\beq\label{eq:radius}
R_0=r+4\Psi(1).
\eeq
\end{lem}
\begin{proof}
Let $F_h$ a minimizing sequence such that
\[
J_{\Lambda_1,\Lambda_2,\alpha}(F_h)\leq \inf_{F\subset \R^n} J_{\Lambda_1,\Lambda_2, \alpha}(F)+\frac{\alpha_1}{h}
\leq P_w(B_r)+\Lambda_2\alpha+\frac{\alpha_1}{h}.
\]
By the second part of Lemma \ref{lem:uniqueness} and from the previous inequality we have 
$$
P_w(B_r)+\Lambda_2||F_h\triangle B_{r}|_w-\alpha|\leq J_{\Lambda_1,\Lambda_2,\alpha}(F_h)
\leq  P_w(B_r)+\Lambda_2\alpha+\frac{\alpha_1}{h}.
$$
In turn this inequality implies that 
\[
|F_h\setminus B_r |_w\leq |F_h\triangle B_r |_w\leq \left(2+\frac{1}{h\Lambda_2}\right) \alpha_1.
\]
Set $\eta:= (\frac{2\Lambda_2+1}{\Lambda_2})\alpha_1<1$.
Thus Lemma \ref{lem:bound1} implies that there exists $r_h\in [r,r+4\Psi(\eta)]$
such that
\[
P_w(F_h \cap B_{r_h})\leq P_w(F_h) - \frac{|F_h\setminus B_{r_h} |_w}{2\Psi(\eta)}
\]
Hence if we set $G_h=F_h\cap B_{r_h} $
we get
\[
\begin{split}
    J_{\Lambda_1,\Lambda_2,\alpha}(G_h)
    \leq& P_w(F_h) - \frac{|F_h\setminus B_{r_h} |_w}{2\Psi(\eta)} +\Lambda_1\left||G_h|_w-|B_r|_w\right|+\Lambda_2||G_h\triangle F_h|_w-\alpha|
    \\
   \leq & J_{\Lambda_1,\Lambda_2,\alpha}(F_h)
     +\left(\Lambda_1-\frac{1}{2\Psi(\eta)}\right)|F_h\setminus B_{r_h} |_w  +\Lambda_2 |G_h \triangle F_h |_w  
    \\
    = &
     J_{\Lambda_1,\Lambda_2,\alpha}(F_h)
     +\left(\Lambda_1 +\Lambda_2-\frac{1}{2\Psi(\eta)}\right)|F_h\setminus B_{r_h} |_w.
\end{split}
\]
Therefore, taking $\eta$, hence $\alpha_1$, sufficiently small we have that $G_h$ is a minimizing sequence such that $G_h \subset B_{R_0}$, where $R_0=$ is as in \eqref{eq:radius}. The conclusion then follows
observing that the sets $G_h$ have all equibounded perimeters and using the well known properties of compactness and lower semicontinuity of the perimeter.
\end{proof}
\begin{lem}  \label{lem:quasimin}
Given $\Lambda_1,\Lambda_2\geq0$,
there exists $\omega\geq0$  such that 
 if $E\subset B_{R_0}$ is a minimizer 
of $J_{\Lambda_1,\Lambda_2,\alpha}$ with $\alpha\geq0$, 
 then $E$ is an $\omega$-minimizer of the perimeter
 in $B_{2R_0}$.
\end{lem}
\begin{proof}
Let $F$ be a set of finite perimeter with $F\triangle E\subset\subset B_{\rho}(x)\subset \subset B_{2R_0}$.
If $|B_\rho(x) \cap E|=0$ then \eqref{eq:disquasi} is trivially satisfied. 

Hence we may assume without loss of generality that $|B_\rho (x)\cap E|>0$.
Since $F\triangle E\subset\subset B_\rho(x)$ we have that 
$P_w(F;\R^n\setminus B_{\rho}(x))=P_w(E,\R^n \setminus B_{\rho}(x))$.
Moreover, 

\[
\left||F|_w-|E|_w\right|\leq e^{w(2R_0)} \omega_n\rho^n.
\]
Similarly,
\[
\left||F\triangle B_r|_w-|E\triangle B_r|_w\right| \leq |F\triangle E|_w
\leq e^{w(2R_0)} \omega_n\rho^n.
\]
The above inequalities and the minimality of $E$ yield 
\beq \label{eq:quasimin1}
\begin{split}
\min_{z\in \overline{ B_\rho}(x)}e^{w(|z|)}P(E;B_\rho(x))\leq& P_w(E;B_\rho(x)) \leq P_w(F,B_\rho(x))+C_0\rho^{n}
\\
\leq &\max_{z\in \overline{ B_\rho}(x)}e^{w(|z|)} P(F,B_\rho(x))+C_0\rho^n
.
\end{split}
\eeq
for a constant $C_0$ depending only on $\Lambda_1,\Lambda_2,r,n$.
Observe now that there exists another constant  $C>0$, still indipendent of $E$, $\alpha_1$ and $\rho$, such that
\beq \label{eq:quasimin2}
P(E,B_\rho(x))\leq C\rho^{n-1}.
\eeq
Indeed, if we first
apply \eqref{eq:quasimin1} with $F$ replaced by $E \cup B_{\rho'}(x)$ with $0<\rho'<\rho$ such that
$\H^{n-1}(\pa^* E\cap \pa B_{\rho'}(x))=0$ and then let $\rho'\uparrow\rho$, we get
\[
\min_{z\in \bar B_{\rho}(x)}e^{w(|z|)}P(E;B_{\rho}(x))\leq n\omega_n \rho^{n-1}\max_{z\in \bar B_\rho(x)}e^{w(|z|)}+C_0\rho^n
\]
which gives \eqref{eq:quasimin2} since $\rho\leq 2R_0$.
Observe also that there exists another constant, still denoted by $C$ and depending only on $R_0$, such that
$$\operatorname*{osc}_{z\in B_\rho(x)} e^{w(|z|)}\leq C \rho.$$
The conclusion easily follows from this estimate using \eqref{eq:quasimin1}
and \eqref{eq:quasimin2}.
\end{proof}
\begin{lem} \label{lem:convergence}
Let $r>0$ such that $w(r)>w(0)$, let $\Lambda_1$, $\Lambda_2$ satisfy the assumptions of Lemma~\ref{lem:uniqueness} and let $\e_i\to 0$.
Let $F_i$ be a sequence of  equibounded minimizers of $J_{\Lambda_1,\Lambda_2,\e_i}$. Then, up to a not relabeled subsequence, $F_i\to B_r$ in $C^{1,\alpha}$ for all $\alpha <\frac12$. Precisely, for all $i$ there exists
  $\psi_i \in C^{1,\frac12}(\S^{n-1})$ such that
\[
\pa F_i= \{rx(1+\psi_i(x)), x \in \S^{n-1}\}\quad \text{ with } \quad\|\psi_i\|_{C^{1,\alpha}(\S^{n-1})} \to 0 \quad \text{ for any } \alpha \in (0,\frac12).
\]

\end{lem}
\begin{proof}
By the minimality of $F_i$ we have 
\[
e^{w(0)}P(F_i)\leq P_w(F_i)\leq J_{\Lambda_1,\Lambda_2,\e_i}(F_i) \leq P_w(B_r) +\e_i.
\]
Since $\e_i\to 0$ and the sets $F_i$ are equibounded, we get that there exists a bounded set of finite perimeter $F $ such that up to a not relabelled subsequence $|F_i\triangle F|\to 0$. Since for any set $E$ of finite weighted perimeter and for every $i\in \mathbb N$ we have
\[
J_{\Lambda_1,\Lambda_2,\e_i}(F_i) \leq J_{\Lambda_1,\Lambda_2,\e_i}(E),
\]
sending $i$ to infinity and using the semicontinuity of the weighted perimeter and the continuity of $\alpha_w(\cdot)$ with respect to the convergence in meausure we infer
\[
J_{\Lambda_1,\Lambda_2}(F) \leq J_{\Lambda_1,\Lambda_2}(E).
\]
Hence, $F$ is a minimizer for the functional $J_{\Lambda_1,\Lambda_2}$ and thus from Lemma \ref{lem:uniqueness} $F_i\to B_r$ in measure. The conclusion then follows from
Lemma \ref{lem:quasimin} and from
 Theorem 
\ref{teo:quasimin2}.

\end{proof}
\begin{proof}[Proof of the Main Theorem]
In order to prove \eqref{eq:mainth} it is enough to show that for any $r>0$ such that $w''(r)>0$ there exists $\delta>0$ such that if $|B_r\triangle E|_w<\delta$ and $|E|_w=|B_r|_w$
then
\beq \label{eq:contradd}
P_w(E)-P_w(B_r)\geq c_1 |B_r\triangle E|_w^2,
\eeq
where $c_1$ is a constant whose explicit value will be given later.
Indeed, by Lemma~\ref{lem:qualitative} there exists $\sigma>0$ such that if $|E\triangle B_r|_w\geq \delta$ then 
$P_w(E)-P_w(B_r)\geq \sigma$ and thus we may conclude that in this case 
\[
P_w(E)-P_w(B_r)\geq \frac{\sigma}{4|B_r|^2_w}|E\triangle B_r|_w^2.
\]
In order to prove \eqref{eq:contradd} we argue by contradiction assuming that there exists a sequence $E_i$ such that $|E_i|_w=|B_r|_w$, $|E_i \triangle B_r|_w\to0$ as $i\to\infty$ 
and 
$$
P_w(E_i)\leq P_w(B_r)+c_1|E_i\triangle B_r|_w^2.
$$
We now set $\e_i= |E_i \triangle B_r|_w$. Let $\Lambda_1\geq  n-1+rw'(r)$
and $\Lambda_2>0$ to be chosen. By Lemma~\ref{lem:existence} we have that
for $i$ sufficiently large the functional
$J_{\Lambda_1,\Lambda_2,\e_i}$ has a minimizer $F_i$ with $F_i\subset B_{R_0}$, $R_0= r+4\Psi(1)$.
Note that by Lemma \ref{lem:convergence}, passing possibly to a subsequence, we have that $F_i\to B_r$ in $C^{1,\alpha}$ for all $\alpha \in (0,1/2)$. 
By the minimality of $F_i$ we have that for
 $i$ large
\beq \label{eq:main1}
J_{\Lambda_1,\Lambda_2,\e_i}(F_i)\leq J_{\Lambda_1,\Lambda_2,\e_i}(E_i)= P_w(E_i) \leq P_w(B_r)+ c_1 \e_i^2.
\eeq
From this inequality, if  $\Lambda_2$ is chosen such that $\Lambda_2> 4(4\frac{n+1}{r}+w'(2r))$, by applying Lemma~\ref{lem:uniqueness} with $\Lambda_2$ replaced by $\Lambda_2/2$ and $\Lambda_1=0$, we have
\[
\begin{split}
P_w(F_i)+\Lambda_2||F_i \triangle B_r|_w-\e_i|&\leq P_w(B_r) + c_1\e_i^2\\
&\leq P_w(F_i) +  \frac{\Lambda_2}{2} |F_i\triangle B_r|_w+c_1\e_i^2
\end{split}
\]
from which it follows that for $i$ large
\beq \label{eq:main3}
|F_i\triangle B_r|_w\geq \frac{\e_i}{2}.
\eeq
Assume now that $\Lambda_1\geq 2(n-1+rw'(r))$. By \eqref{eq:main1} and Lemma~\ref{lem:uniqueness} with $\Lambda_1$ replaced by $\Lambda_1/2$ and $\Lambda_2=0$ we have
\[
\begin{split}
P_w(F_i)+\Lambda_1 ||F_i|_w-|B_r|_w| &\leq P_w(B_r)+c_1 \e_i^2
\\
&\leq P_w(F_i)+\frac{\Lambda_1}{2}||F_i|_w-|B_r|_w|+c_1\e_i^2.
\end{split}
\]
From this in particular we deduce that
\beq \label{eq:main30}
||F_i|_w- |B_r|_w |\leq 2c_1 \e_i^2.
\eeq
Denote by $r_i$ the radius such that $|B_{r_i}|_w=|F_i|_w$. From the inequality above we have
\[
|F_i\triangle B_{r}|_w\leq |F_i\triangle B_{r_i}|_w+ |B_r\triangle B_{r_i}|_w
\leq | F_i \triangle B_{r_i}|_w +2c_1 \e_i^2 
\]
and thus for $i$ large, using  \eqref{eq:main3}, we have
$$|F_i\triangle B_{r}|_w \leq 2|F_i\triangle B_{r_i}|_w. $$
In turn, this inequality together with \eqref{eq:main30} and \eqref{eq:main3} implies
\[
\begin{split}
P_w(B_r)\leq P_w(B_{r_i}) +C |r-r_i|
\leq P_w(B_{r_i}) +Cc_1\e_i^2
\leq P_w(B_{r_i})+\tilde C c_1 |F_i\triangle B_{r_i}|_w^2.
\end{split}
\]
which is a contradiction to \eqref{eq:nearlysph} if we choose $c_1<c_0/\tilde C$, where $c_0$ is the constant provided by Proposition \ref{lem:nearlysph}.

\end{proof}

\section{Negative power weights}
Given a measurable set $E\subset\R^n$ and $a\in[0,1]$ we denote by $E^{(a)}$ the set of points in $\R^n$ where $E$ has density equal to $a$, that is the set of points $x\in\R^n$ such that
$$
\lim_{r\to0}\frac{|E\cap B_r(x)|}{\omega_nr^n}=a\,.
$$
Note that $E^{(1)}$ and $E$ coincide up to a set of zero measure.
If $E\subset\R^n$ is a set of locally finite perimeter and $p\in\R$ we set
$$
P_p(E)=\int_{\partial^*E}|x|^p\,d\mathcal H^{n-1}.
$$
As already mentioned in the introduction it is well known that for $p>0$ the only isoperimetric sets with respect to the weight $|x|^p$ are balls centered at the origin and moreover they are stable. 
On the contrary, when $1-n<p<0$ there are no isoperimetric sets. We recall also that if $p\leq 1-n$ the isoperimetric inequality
\beq \label{eq:isopp}
P_p(E)\geq P_p(B_r)
\eeq
where $|E|=|B_r|$ is true whenever $E$ is an open set containing the origin (see for instance \cite{Csa}).
The condition $0\in E$ is clearly necessary for \eqref{eq:isopp} to holds since $P_p(B_r(x))\to 0$ as $|x|\to \infty$. Next theorem shows that a quantitative version of \eqref{eq:isopp} is also true.
\begin{thm} 
Let $n\geq 2$ and $p<-n-1$. There exists a constant $c>0$ such that if $r>0$ and $E$ is  a set of finite perimeter with $|E|=|B_r|$ such that the origin belongs to the interior of $E^{(1)}$, then 
\[
P_p(E)\geq P_p(B_r)+c |E\triangle B_r|^2.
\]
\begin{proof}
Since $0\in \operatorname{int}( E^{(1)})$
 that there exists $r_0>0$ such that $E^{(1)}\cap B_{r_0}=B_{r_0}$. 
Thus, we note that
\[\begin{split}
P_p(E)-P_p(B_r)&=\int_{\pa^* E}|x|^p\,d\H^{n-1}-\int_{\pa B_r}|x|^p\,d\H^{n-1}\\
&\geq 
\int_{\pa^* E}|x|^p\langle \frac{x}{|x|},\nu_E\rangle \,d\H^{n-1}- \int_{\pa B_r}|x|^p\langle \frac{x}{|x|},\nu_{B_r}\rangle \,d\H^{n-1}
\\
&=\int_{\pa^*( E\setminus B_{r_0})}|x|^p\langle \frac{x}{|x|},\nu_E\rangle \,d\H^{n-1}- \int_{\pa^* (B_r\setminus B_{r_0})}|x|^p\langle \frac{x}{|x|},\nu_{B_r}\rangle \,d\H^{n-1}.
\end{split}
\]
By applying the divergence theorem to $E\setminus B_{r_0}$ and $B\setminus B_{r_0}$ we get
\[\begin{split}
P_p(E)-P_p(B_r)& \geq (n-1+p)\left(\int_{E\setminus B_{r_0}} |x|^{p-1} \, dx -\int_{B_r\setminus B_{r_0}} |x|^{p-1}\, dx \right)\\
&\geq (n-1+p)\left(\int_{B_{\bar r}\setminus B_{r}} |x|^{p-1} \, dx  \right)
\end{split}
\]
where $\bar r$ is such that $|B_{\bar r}\setminus B_r|=|E\setminus B_{r}|$.
The conclusion then follows as in Lemma 6.1 in \cite{LAM1}.
\end{proof}
\end{thm}
As mentioned before inequality \eqref{eq:isopp} does not hold if $0 \not\in \R^{n}\setminus \overline{E^{(1)}} $. Thus in order to get a complete picture it remains to analyze the case $0\in \pa E^{(1)} $. To this aim we recall that if $E$ is a set of locally finite perimeter then
\begin{equation}\label{uno}
\overline{\partial^*E}=\partial E^{(1)}\,.
\end{equation}
The above inequality is well known to the experts, however for the reader's convenience we provide its simple proof.
Recall that given any set of locally finite perimeter $E$ the reduced boundary $\partial^*E$ is always contained in the topological boundary $\partial E$ and does not change if one modifies $E$ by a set of zero Lebesgue measure. Therefore
\begin{equation}\label{due}
\overline{\partial^*E}=\overline{\partial^*E^{(1)}}\subset\partial E^{(1)}\,.
\end{equation}
To show the opposite inclusion, let $x\not\in\overline{\partial^*E}$. Then, there exists $B_r(x)$ such that $\partial^*E\cap B_r(x)=\emptyset$. Thus $P(E;B_r(x))=\mathcal H^{n-1}(\partial^*E\cap B_r(x))=0$. Then by the relative isoperimetric inequality in a ball we have
$$
\min\{|E\setminus B_r(x)|,|E\cap B_r(x)|\}^{\frac{n-1}{n}}\leq c(n)P(E;B_r(x))=0\,.
$$
Therefore, if $|E\setminus B_r(x)|=0$ then $B_r(x)\subset E^{(1)}$ and so $x$ belongs to the interior of $E^{(1)}$. If instead $|E\cap B_r(x)|=0$, then $B_r(x)\subset E^{(0)}$ and so $x$ is in the interior of $\R^n\setminus E^{(1)}$. In both cases $x\not\in\partial E^{(1)}$. 
Therefore, recalling \eqref{due} we get \eqref{uno}.
\begin{lem}\label{remrem}
Let  $p<-n+1$ and $E$ a set of finite perimeter. If 
$0 \in \pa^*E$ then
$P_p(E)= \infty.$

\end{lem}
\begin{proof}
Assume by contradiction that
$P_p(E)<\infty$. Given a ball $B_r$ centered at $0$ we would have
$$
P(E;B_r)r^p\leq\int_{\partial^*E\cap B_r}|x|^p\,d\mathcal H^{n-1}\leq P_p(E)
$$
and thus 
$$
\frac{P(E;B_r)}{r^{n-1}}\leq\frac{P_p(E)}{r^{n-1+p}}\,.
$$
Since $n-1+p<0$, from this inequality we get
$$
\lim_{r\to0}\frac{P(E;B_r)}{r^{n-1}}=0
$$
thus $0\not\in\partial^*E$ which is a contradiction.
\end{proof}
It remains to examine the case $0\in\partial E^{(1)}\setminus \pa^* E$. Next example shows that in this case the isoperimetric inequality may be false.
\begin{example}
Fix  $\alpha>1$. Let $\{p_h\}_{h=0,1,\dots,}$ a dense sequence in $B_1$, with $p_h\not=0$ for all $h$ and set 
$$
r_i=\frac{r}{2^{\frac{i}{n}}}\quad\text{for all $i=0,1,\dots$, with}\,\, 0<r<\frac{1}{2^\alpha}
$$ 
to be chosen later. We now rearrange the elements of the sequence $\{p_h\}$ as follows. First we set $q_0=p_{h_0}$, where $h_0$ is the smallest integer such that $|p_{h_0}|^\alpha>2^\alpha r_0$. Notice that this is always possible since $2^\alpha r_0<1$. Then, for all $i=1,2,\dots$ we set $q_i=p_{h_i}$, where $h_i$ is the smallest integer different from $h_0, h_1,\dots, h_{i-1}$ such that 
\begin{equation}\label{tre}
|q_i|^\alpha=|p_{h_i}|^\alpha>2^\alpha r_i\,.
\end{equation}
Notice that since $r_i\to0$ as $i\to\infty$, all the elements of the sequence $\{p_h\}$ will be chosen once and only once. Finally we set 
$$
E=\bigcup_{i=0}^\infty B_{r_i}(q_i)\,.
$$
Then $|E|\leq \sum_{i}|B_{r_i}(q_i)|=2\omega_nr^n$. Therefore, if $B_E$ is the ball centered at the origin such that $|E|=|B_E|$ we have
$$
P_p(B_E)\geq P_p(B_{2^{\frac{1}{n}}r})=2^{\frac{n-1+p}{n}}n\omega_nr^{n-1+p}\,.
$$
Observe now that if $x\in\partial B_{r_i}(q_i)$ then $|x|\geq|q_i|-r_i$ and by \eqref{tre} $|q_i|-r_i>2r_i^{\frac{1}{\alpha}}-r_i>r_i^{\frac{1}{\alpha}}$.
Therefore for all $i$
$$
\int_{\partial B_{r_i}(q_i)}|x|^p\,d\mathcal H^{n-1}\leq n\omega_nr_i^{n-1+\frac{p}{\alpha}}
$$
and thus
$$
P_p(E)\leq \sum_{i=0}^\infty\int_{\partial B_{r_i}(q_i)}|x|^p\,d\mathcal H^{n-1}\leq C(n,\alpha,p)r^{n-1+\frac{p}{\alpha}}<n\omega_n2^{\frac{n-1+p}{n}}r^{n-1+p}\leq P_p(B_E)\,,
$$
provided $r$ is sufficiently small. 

Let us now show that $0\in\overline{\partial^*E}$.
Since $B_1\subset E^{(1)}\cup  \pa E^{(1)}$ it is enough to show that $0\in E^{(0)}$. To this end we estimate for $0<\varrho<r$ 
\begin{equation}\label{quattro}
\frac{|E\cap B_\varrho|}{\omega_n\varrho^n}\leq \frac{1}{\varrho^n}\sum_{\{i:\,B_{r_i}(q_i)\cap B_\varrho\not=\emptyset\}}r_i^n\,.
\end{equation}
Note that  if $B_{r_i}(q_i)\cap B_\varrho\not=\emptyset$ then $|q_i|-r_i<\varrho$ and thus, recalling \eqref{tre}, $r_i<\varrho^\alpha$. Thus, from \eqref{quattro} we have, denoting by $\lfloor\cdot\rfloor$ the integer part of a real number,
\begin{align*}
\frac{|E\cap B_\varrho|}{\omega_n\varrho^n}&\leq  \frac{r^n}{\varrho^n}\sum_{\{i:\,2^{\frac{i}{n}}>r/\varrho^\alpha\}}\frac{1}{2^i}= \frac{r^n}{\varrho^n}\sum_{i=1+\lfloor n\log_2(r/\varrho^\alpha)\rfloor}^\infty\frac{1}{2^i}\\
&= \frac{r^n}{\varrho^n}\frac{1}{2^{\lfloor n\log_2(r/\varrho^\alpha)\rfloor}}\leq\frac{r^n}{\varrho^n}\frac{2}{2^{ n\log_2(r/\varrho^\alpha)}}=2\frac{\varrho^{n\alpha}}{\varrho^n}\,.
\end{align*}
Then we conclude that
$$
\lim_{r\to0}\frac{|E\cap B_\varrho|}{\omega_n\varrho^n}=0\,,
$$
thus proving that $0\in\overline{\partial^*E}$. Finally observe that, thanks to Remark~\ref{remrem} we have indeed that $0\in\overline{\partial^*E}\setminus\partial^*E$.
\end{example}
  \section{Aknowledgment}
The authors wish to thank Adi Adimurthi for informing them of a mistake in a preliminary version of the paper.\\
This research was supported by MUR project PRIN2017TEXA3H. 

\bibliographystyle{plain}
\bibliography{references}

\end{document}